\documentclass[twoside,centertags]{amsart}

\usepackage[utf8]{inputenc}         % for utf8 encoding to avoid tex accents, for example
\usepackage[T1]{fontenc}            % for T1 out encoding
\usepackage{a4wide}                 % for a4 paper size
\usepackage{amsmath, amssymb}       % various ams packages
\usepackage{amsthm}                 % ams theorem styles

\usepackage{tikz}                   % graphics and diagrams
\usepackage{tikz-cd}                % commutative diagrams
\tikzcdset{ % for the style of arrow tips
arrow style=tikz
}

\usepackage{microtype}
\usepackage{color}                  % colors for our comments, for example, and for links
\usepackage[normalem]{ulem}         % long underline across lines 
\usepackage{enumitem}               % configure appearance of lists and references
\setlist[itemize]{leftmargin=*}
\setlist[enumerate]{leftmargin=*,label=\roman*),ref=\roman*)}
\newlist{subenumerate}{enumerate}{2}
\setlist[subenumerate]{leftmargin=*,label=\alph*),ref=\alph*)}

\usepackage{graphicx}               % check if used ??
\usepackage{stackrel}               % improved stacking of things

\usepackage{blkarray}               % for arrays and tables ?? should go

\usepackage{hyperref}

\usepackage{bbm}                    % blackboard style fonts 
\usepackage{stix}                   % just used for nsimeq 
\usepackage{eucal}                 
\usepackage{xypic}

\newcommand{\defi}[1]{\emph{#1}}                 % highlight or emphasize a word \emph{that is being defined} (not just any emphasis).

\newcommand{\pullbacksign}{%
	\hspace{-0.25ex}%
	\tikz[baseline=(pb.base)]{%
		\draw[line width=rule_thickness, line cap=round]%
		(0,0)%
		++ (-0.5ex,-0.5ex)%
		-- ++(1ex,0ex)%
		-- ++ (0ex,1ex);%
		\filldraw (-0.3ex,0.3ex) circle (0.3pt);%
		\node (pb) at (0,0) {\phantom{x}};%
	}%
}
\newcommand{\pushoutsign}{%
	\hspace{0.25ex}%
	\tikz[baseline=(po.base)]{%
		\draw[line width=rule_thickness, line cap=round]%
		(0,0)%
		++ (-0.5ex,-0.5ex)%
		-- ++(0ex,1ex)%
		-- ++ (1ex,0ex);%
		\filldraw (0.3ex,-0.3ex) circle (0.3pt);%
		\node (po) at (0,0) {\phantom{x}};%
	}%
}
\tikzcdset{
	pullback/.style = {dash, phantom, "\pullbacksign", start anchor=center, end anchor=center},
	pushout/.style = {dash, phantom, "\pushoutsign", start anchor=center, end anchor=center}
}

\newcommand{\Hom}{\operatorname{Hom}}            % Hom functor \warn{set Hom?}
\newcommand{\nerv}{\operatorname N}              % Rezk nerve
\newcommand{\Nrel}{\nerv^{\mathrm{rel}}}         % relative nerve
\newcommand{\op}{^\mathrm{op}}                   % opposite category
\newcommand{\pop}[1]{^{\mathrm{op}_{#1}}}

\newcommand{\Ar}{\operatorname{Ar}}              % arrow cat

\newcommand{\Twar}{\operatorname{TwAr}}          % twisted arrow cat
\newcommand{\TwarD}{\operatorname{TwArD}}          % twisted arrow double cat
\newcommand{\ArD}{\operatorname{ArD}}           %arrow double category
\newcommand{\Sq}{\operatorname{Sq}}             % double category of squares

\newcommand{\id}{\mathrm{id}}                % identity
\newcommand{\const}{\mathrm{const}}          % constant function

\newcommand{\An}{\mathrm{An}}
\newcommand{\Set}{\mathrm{Set}}                  % sets
\newcommand{\Cat}{{\mathrm{Cat}_\infty}}                  % small infinity categories
\newcommand{\DCat}{\mathrm{DCat}_\infty}                  % small infinity double categories
\newcommand{\DSeg}{\mathrm{DSeg}} 			%double Segal spaces
\newcommand{\RelCat}{\mathrm{Cat}_\infty^{\mathrm{rel}}}

\newcommand{\sSet}{\mathrm{s}\Set}
\newcommand{\ssSet}{\mathrm{ss}\Set}
\newcommand{\sAn}{\mathrm{s}\An}
\newcommand{\ssAn}{\mathrm{ss}\An}
\newcommand{\Seg}{\operatorname{Seg}}

\newcommand{\Span}{\operatorname{Span}}          % Span category
\newcommand{\Cob}{\mathrm{Cob}}                  % Cobordism category
\newcommand{\Fun}{\operatorname{Fun}}            % functor cat
\newcommand{\diag}{\mathrm{Diag}}                         % decoration for restriction alng diagonal
\newcommand{\oSun}{\oS^{\mathrm{un}}}

\newcommand{\K}{\operatorname K}                 % K-Theory 
\newcommand{\Q}{\operatorname{Q}}                % Q-construction
\newcommand{\oS}{\operatorname S}                % S-construction
\newcommand{\os}{\mathrm{s}}
\newcommand{\Stair}{\mathrm{Step}}
\newcommand{\cat}{\mathrm{cat}}
\newcommand{\Spine}{\mathrm{Spine}}
\newcommand{\SpineD}{\mathrm{SpineArD}}
  
\newcommand{\cart}{\mathrm{cart}}
\newcommand{\cocart}{\mathrm{cocart}}
\newcommand{\co}{\mathrm{co}}

\newcommand{\C}{\mathcal C}                
\newcommand{\D}{\mathcal{D}}               

\newtheoremstyle{thms}
	{}{}{\itshape}{}{\bfseries }{}{ }
	{\thmname{#1} \thmnumber{#2}. \thmnote{\bfseries{(#3)}}}

\newtheoremstyle{defs}
	{}{}{\normalfont}{}{\bfseries }{}{ }
	{\thmname{#1} \thmnumber{#2}. \thmnote{\bfseries{(#3)}}}

\newtheoremstyle{rmk}
	{}{}{\normalfont}{}{\itshape }{}{ }
        {}

\theoremstyle{thms}
\newtheorem{proposition}{Proposition}[section]
\newtheorem{theorem}[proposition]{Theorem}
\newtheorem{lemma}[proposition]{Lemma}
\newtheorem{corollary}[proposition]{Corollary}

\theoremstyle{defs}

\newtheorem{remark}[proposition]{Remark}

\theoremstyle{defs2}

\theoremstyle{rmk}

\author{George Raptis}
\address{Department of Mathematics, Aristotle University of Thessaloniki, 541 24 Thessaloniki, Greece}
\email{raptisg@math.auth.gr}

\author{Wolfgang Steimle}
\address{Institut f\"ur Mathematik, Universit\"at Augsburg, Germany}
\email{wolfgang.steimle@math.uni-augsburg.de}

\title{The span-squares adjunction}

\date{\today}
\begin{document}

\begin{abstract}
We show a universal property of the span $\infty$-category that yields a description of functors defined on this category. For this, we view  the span construction as a functor from double $\infty$-categories to $\infty$-categories, and show that this functor admits a right adjoint defined by the double $\infty$-categories of squares. As an application, we obtain a model for the localization of an $\infty$-category in terms of a span $\infty$-category and also give another proof of the model given by the relative nerve. Moreover, using the span-squares adjunction, we obtain new proofs of the equivalences between different models of algebraic $K$-theory, given by the $\Q$-, the $\oS$-, the cobordism model, and the squares construction. 
\end{abstract}
\maketitle

\section{Introduction}

The span category $\Span(\C)$ of an ($\infty$-)category $\C$ is an important and useful device for expressing combined covariant and contravariant functorialities that are related by base-change formulas. Such situations arise notably in the well-studied context of 6-functor formalisms (where the contravariant functor is given by pull-back of sheaves and the covariant functor extends proper pushforward), see e.g.\ \cite{CCL2,Man} for recent treatments. They arise also in equivariant homotopy theory (where the contravariant functoriality is given by transfer and the base-change formula expresses the Mackey double coset formula), see e.g.\ \cite{BDA, Bar3}, 
and in the context of higher semiadditivity (where the two functorialities combine to express norm maps), see e.g.\ \cite{CCL}.

In such situations, a functor $\Span(\C)\to \D$ into another $\infty$-category $\D$ should exactly integrate these combined covariant and contravariant data in a uniform and orderly way. As usual in higher category theory, the compatibilities between covariant and contravariant structures form a system of higher coherences. Thus, it becomes a complex problem to formulate in a useful way the precise data required for the definition of a functor out of the span $\infty$-category. Motivated by a theorem of Liu--Zheng \cite[Theorem 1.4.26]{LZ}, the first main goal of this note is to establish a simple such description, formulated in terms of an adjunction (\emph{span-squares adjunction}) between $\infty$-categories and \emph{double $\infty$-categories}. For this purpose, it is crucial to regard the span category as an $(\infty,1)$-category that is functorially associated to a double $\infty$-category (or even a bisimplicial space); this is not to be confused with enhancements of the span category itself to an $(\infty,2)$-category. 

The span-squares adjunction also extends to the context of relative categories, 
providing an embedding of these into double categories. As we shall explain, this relative version of the adjunction can be viewed as strengthening of Rezk's relative nerve theorem, and in addition shows that span categories, in our generality, provide a model for $\infty$-categorical localization.
% and, as an application, we show that the universal property of the span $\infty$-category yields a model for $\infty$-categorical localization.

Span ($\infty$-)categories and their geometric realizations also arise in algebraic $K$-theory (via Quillen's $\Q$-construction) \cite{Qui, Bar}. Our double-categorical setup has the pleasant feature that one may take opposites in either of the two directions alone. It turns out that this procedure interchanges the $\Q$-construction with (an unpointed, double- and $\infty$-categorical extension of) the Segal-Waldhausen $\os$- and $\oS$-construction \cite{Wal}. The second main goal of this note is to 
draw conclusions of the span-squares adjunction concerning different models of algebraic $K$-theory. Specifically, we obtain natural equivalences between the geometric realizations 
\[ \vert \Q(\D)\vert \simeq \vert \oS(\D)\vert \simeq \vert \D\vert\]
for any double $\infty$-category $\D$. This gives new proofs and generalizations of the equivalences between various different models of algebraic $K$-theory \cite{Qui, Wal, Bar, RS, CKMZ}. The right-hand term can be considered as a double $\infty$-categorical version of squares $K$-theory (see \cite{CKMZ}); this version of algebraic $K$-theory has been considered for the $K$-theory of varieties \cite{Cam} and for the scissors congruence $K$-theory of manifolds \cite{HMMRS, MRS}. 

\medskip 

Let us now explain the span-squares adjunction in more detail. We denote by $\Cat$ the $\infty$-category of small $\infty$-categories and by $\DSeg$ the $\infty$-category of (what we shall call) double $\infty$-categories and by which we mean explicitly double Segal spaces; see Remark \ref{rem:completeness} for a discussion of completeness conditions. We can informally describe the span category\footnote{For the sake of brevity, we will omit the prefix $\infty=(\infty,1)$ from now on.} of a double category $\D$ as follows (see Section \ref{sec:preliminaries} for precise definitions): an object of $\Span(\D)$ is an object ( = $(0,0)$-simplex) of $\D$, and morphisms in the span category are generated by spans of the form
\[ d_0 \xleftarrow{f} d_{01} \xrightarrow{g} d_1\]
where $f$ is a horizontal morphism ( = $(0,1)$-simplex) of $\D$ and $g$ is a vertical morphism ( = $(1,0)$-simplex), to be thought of as morphisms from $d_0$ to $d_1$. The commutative triangles in $\Span(\D)$ are generated by the diagrams of shape
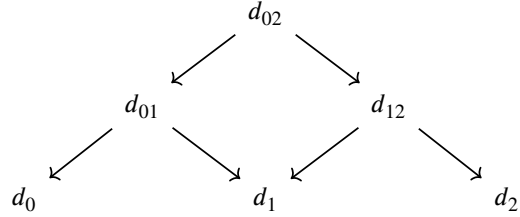
\[
\begin{tikzcd}[sep=tiny]
    && d_{02} \ar[ld]\ar[rd]\\
    & d_{01} \ar[ld] \ar[rd] && d_{12} \ar[ld]\ar[rd]\\
    d_0 && d_1&& d_2
\end{tikzcd}
\]
where the left-pointing maps are horizontal morphisms of $\D$, the right-pointing maps are vertical morphisms of $\D$, and the upper square is a square in the double-categorical structure of $\D$. A straightforward extension to higher coherence data provides a functor $\Span \colon \DSeg \to \Cat$. The usual span category of a category and known variants thereof, such as for (adequate) triples, are special cases of this construction on a suitable double category.

Conversely, we have a functor which sends a category to its \emph{double category of squares}, defined via
\[ \Sq\colon \Cat \to \DSeg, \quad  \Sq(\C)_{m,n}:=\Hom_\Cat([m]\times [n], \C).\]
Thus, both the horizontal and the vertical morphisms of $\Sq(\C)$ are  given by the morphisms of $\C$, and the squares of the double category $\Sq(\C)$ are just the commutative squares in $\C$. Then we have:

\begin{theorem}[The span-squares adjunction]\label{span-constr}
The functor $\Span\colon \DSeg \to \Cat$ is left adjoint to the functor $\C\mapsto \Sq(\C)\pop2$. Moreover, $\Sq(-)\pop2 \colon \Cat \to \DSeg$ is fully faithful.
\end{theorem}

Here, $(-)\pop2$ denotes the opposite in the second/horizontal category direction. Through this adjunction, a functor $f\colon \Span(\D)\to \C$ is uniquely determined by an associated functor of double categories $\D \to \Sq(\C)\pop2$, that is, a structure that consists of a functor $f_v \colon \D_v \to \C$ from the vertical category of $\D$ to $\C$, a (contravariant) functor $f_h\colon \D_h\op\to \C$ from the horizontal category of $\D$, and coherent choices of commutative squares in $\C$ that are functorially associated (both vertically and horizontally) to the squares in $\D$. This is clearly related to the type of structure that emerges in the context of the six functors, however, we will not pursue this relation to formalisms or constructions of such structures in this paper.

The span category can be defined similarly for arbitrary bisimplicial spaces. Our proof of Theorem \ref{span-constr} is based on an identification of the span category in this more general setting. Specifically, for any bisimplicial space $X \in \ssAn$, there is a canonical comparison map of simplicial spaces
$$w_X \colon \delta(X) \to \Q(X\pop2) \cong \oS(X)$$
where $\delta(X) \in \sAn$ denotes the diagonal simplicial space of $X$, and $\Q(X) \in \sAn$ (the \defi{Q-construction} of $X$) is the simplicial space whose associated category defines the span category of $X$. We show in Theorem \ref{S-constr} that $w_X$ induces an equivalence between the associated categories. Thus, the categorical realization $\diag(X) \in \Cat$ of $\delta(X)$ (the \emph{diagonal category} of $X$) agrees with the span category of $X\pop2$. This allows us to interpret the span category as a concrete, often more computable, model for the diagonal category.  

In the course of the proof of Theorem \ref{span-constr}, we will also show that the functor $\Span$ has a contractible space of endomorphisms.

Theorem \ref{span-constr} also extends to relative categories. For any relative category $(\C, W)$, the associated double category $\Sq(\C, W)$ is the double subcategory of $\Sq(\C)$ whose horizontal morphisms are required to lie in $W$. Moreover, $\diag$ (or $\Span$) can be equipped with a natural subcategory, so that it gives rise to a functor $\DSeg \to \RelCat$, which is again left adjoint 
to the fully faithful functor $\Sq \colon \RelCat \to \DSeg$ (Proposition \ref{prop:basic_relative_adjunction}). As a corollary, we will show that the span category $\Span \Sq(\C, W)$ is a model for the localization $\C[W^{-1}]$, and also recover the identification of the latter with the relative nerve of $(\C, W)$ shown in \cite{MG, AC} (Theorem \ref{thm:loc_as_span}).

\begin{remark}\label{rem:completeness}
Theorem \ref{span-constr} remains true after passing from  $\DSeg$ to its full subcategory spanned by bicomplete double Segal spaces, and also after passing to double Segal spaces that are complete in only the first (or second) direction. This claim is a formal consequence of the fact that the functor $\Sq$ takes values in bicomplete double Segal spaces. We mention that, conversely, the span-squares adjunction extends to an adjunction between $\Cat$ and the category $\ssAn$ of bisimplicial spaces and, in fact, to an adjunction between the category $\Seg$ of Segal spaces and the category $\ssAn$, see Remark \ref{Seg-rem3} below. 
% In a different direction, we mention that the map $w_X$ above is an equivalence already after Segal space reflection. It follows that there is a version of the span-squares adjunction between $\Seg$, the category of Segal spaces, and $\DSeg$ (see Remarks \ref{Segal-rem1}, \ref{Segal-rem2} and \ref{Seg-rem3}).
\end{remark}

\subsection*{Relation to other work}
Our Theorem \ref{S-constr} is closely related to and, in fact, extends a result of Liu--Zheng \cite[Theorem 1.4.26]{LZ}; we present here a different conceptual approach and offer an alternative and more categorical perspective. Apart from this, subsequent literature contains a number of universal properties of span categories. % , which are however \emph{not} directly related to our result. 
Specifically, \cite{HHLN} gives a description of maps \emph{into} span categories in terms of an adjunction, while a number of other works (e.g., \cite{GR, Mac, Ste, CCL2}, \cite{Hau}) treat the span category crucially as an $(\infty, 2)$ or even $(\infty,n)$-category, and involve a different setup in general. However, special cases of our results, i.e., for suitable double categories that arise from squares categories, recover special cases of the universal property of the span $(\infty,2)$-category studied by Gaitsgory-Rozenblyum \cite[III.7, 2.1.3]{GR}.

In a different direction, for certain specific input ($\infty$-)categories $\C$, the span $(\infty,1)$-category $\Span(\C)$ also has a universal property in the context of ambidexterity: most prominently, the span category $\Span(\mathrm{Fin})$ of finite sets is the free semi-additive $\infty$-category on one object \cite{Harpaz}. Our double-categorical version of the $\oS$-construction has also been considered in \cite{Juran} in the context of orthogonal factorization systems (see Remark \ref{rem:juran}). 

\subsection*{Organization of the paper}
The main definitions are given in Section \ref{sec:preliminaries}, along with a discussion of their basic properties. It ends with a reformulation of our main Theorem (Theorem \ref{S-constr}), the proof of which occupies the next two Sections \ref{sec:arrow_double} and \ref{proof-main-thm}. We discuss the extension to relative categories and two models for the localization in Section \ref{localizations}. The final Section \ref{double_vs_single_span} is devoted to the applications in algebraic $K$-theory. 

\subsection*{Notations and conventions}
We denote by $\An$ the category of spaces ( = animae) and consider it as a full subcategory of $\Cat$, the category of (small) categories. We also identify the latter category, through the Rezk nerve, with a full subcategory of the category $\sAn=\Fun(\Delta\op, \An)$ of simplicial spaces. For example, we will usually identify $[n] \in \Cat$ with $\Delta^n \in \sAn$ in the notation. We refer to the left adjoints 
\[ \cat\colon \sAn \to \Cat \quad \mathrm{and} \quad \vert - \vert \colon \sAn \to \An \]
of the inclusions as \defi{categorical} and \defi{geometric realization} respectively; the latter functor is just the colimit over $\Delta\op$ and extends the geometric realization of categories. A map of simplicial spaces is called \defi{categorical equivalence} if it becomes an equivalence after categorical realization (and therefore, after geometric realization). 

As already explained above, we consider the category of (what we call here) double categories $\DSeg$ as the full subcategory of the category $\ssAn:=\Fun(\Delta\op \times \Delta\op, \An)$, consisting of the bisimplicial spaces that satisfy the Segal condition pointwise in both directions 
(without any additional completeness assumptions). We refer to the left adjoint $\ssAn \to \DSeg$ of the inclusion $\DSeg \subset \ssAn$ as \defi{double-categorical realization} functor. $\DSeg$ is also a full reflective subcategory of the full reflective subcategory of $\ssAn$ that is spanned by the bisimplicial spaces which satisfy the Segal condition pointwise in only the first (or second) direction. Hence the double-categorical realization functor factors through the reflection into either of these two intermediate full subcategories. We write $[n] \otimes [m]$ for the double category that is represented by $([n], [m])\in \Delta\times \Delta$. 

Following the standard convention from matrix theory, we will refer to the first (resp., second) simplicial direction in a bisimplicial space as the \defi{vertical} (resp., \defi{horizontal}) one and draw simplices in diagrams accordingly. 

\subsection*{Acknowledgements} We thank Bastiaan Cnossen and Fabian Hebestreit for helpful comments. G.R. was supported by the Hellenic Foundation for Research and Innovation (H.F.R.I.) under the ``3rd Call for H.F.R.I.’s Research Projects to Support Faculty Members \& Researchers'' (Project Number: 25480).  W.S. was partially supported by the CRC~1785
    \emph{Generalised Motivic Methods in Geometry}
    (Universit\"at Regensburg, funded by the DFG).

\subsection*{LLM declaration} No Artificial Intelligence tool was used in the mathematical research or writing of the present paper.

%%%%%%%%%%%%%%%%%%%%%%%%%%%%%%%%%%%%%%%%%%%%%%%%%%%%%%%%%%
%%%%%%%%%%%%%%%%%%%%%%%%%%%%%%%%%%%%%%%%%%%%%%%%%%%%%%%%%%%%%

\section{Models of the diagonal category functor}\label{sec:preliminaries}

\subsection{The diagonal category}

The \defi{diagonal simplicial space} $\delta(X)$ of a bisimplicial space $X\in \ssAn$ is defined as the restriction of $X$ along the diagonal functor $\Delta\op \to \Delta\op \times \Delta\op$, and the \defi{diagonal category} $\diag(X)$ is its categorical realization:
\[\diag\colon \ssAn \xrightarrow{\delta} \sAn \xrightarrow{\cat} \Cat.\]
Informally, say for $X=\D$ a double category, the diagonal category $\diag(\D)$ is the category whose morphisms are generated by the squares in $\D$. We will frequently use the computation 
\[\diag([m]\otimes [n])= [m]\times [n]\]
which stems from the equivalence of simplicial spaces $\delta([m]\otimes [n])=[m]\times[n]$.

\begin{proposition}\label{prop:basic_adjunction}
The diagonal category functor is part of an adjunction 
\begin{equation} \label{Diag-adjunction} \tag{$\diag$}
\diag \colon \DSeg \rightleftarrows \Cat \colon \Sq
\end{equation}
where $\diag$ is the left adjoint. Moreover, the counit transformation 
$\diag\Sq(\C)\to \C$ is an equivalence (so that $\Sq$ identifies $\Cat$ with a Bousfield localization of $\DSeg$). 
\end{proposition}

\newcommand{\Fin}{\mathrm{Fin}}

\begin{remark}
Both functors in this adjunction preserve products and so we also obtain an induced adjunction between symmetric monoidal categories and symmetric monoidal double categories. 
\end{remark}

\begin{proof}
  We first observe that the functor $\delta$ is part of an adjunction 
$$\delta \colon \ssAn \rightleftarrows \sAn \colon \Sq$$
where the right adjoint $\Sq$ is the obvious extension of $\Sq\colon \Cat\to \DSeg$, defined by the same rule for any $X \in \sAn$,
\[ \Sq(X) \colon \Delta\op \times \Delta\op \to \An, \ \Sq(X)_{m,n}:=\Hom_{\sAn}([m]\times [n], X).\]
As $\Sq$ sends $\Cat$ to $\DSeg$, we obtain an induced adjunction between $\Cat$ and $\DSeg$, as claimed. 

For the second statement, we note that the counit transformation of the adjunction $(\delta, \Sq)$ is given on $n$-simplices by the map 
\[d^*\colon (\delta \Sq(X))_n = \Hom_{\sAn}([n]\times [n], X) \to \Hom_{\sAn}([n], X) = X_n\]
that restricts along the diagonal $d\colon [n]\to [n]\times [n]$. A map in the other direction is given by precomposition with the minimum function $m\colon [n]\times [n]\to [n]$, thought of as a natural transformation of functors $\Delta\to \sAn$. Then $d^* \circ m^* = \id$ since $m\circ d=\id$. On the other hand, the natural transformation $H\colon d\circ m \rightarrow \id$ of endofunctors on $[n]\times [n]$, thought of as a natural transformation $[n]\times [n]\times [1]\to [n]\times [n]$ of functors $\Delta\to \sAn$, gives rise to a natural transformation $m^*\circ d^* \rightarrow\id_{\delta\Sq(X)}$. Explicitly, this is defined by the simplicial map $\delta \Sq(X) \times [1] \to \delta \Sq(X)$ that is given on $n$-simplices by the rule
 \[( \sigma\colon [n]\times [n]\to X,\; \alpha\colon [n]\to [1]) \mapsto ([n]\times [n]\xrightarrow{(\id, \alpha\circ m)} [n]\times [n]\times [1]\xrightarrow{H} [n]\times [n]\xrightarrow \sigma X).\]
After applying categorical realization (which commutes with products), and evaluating at $X=\C\in \Cat$, we obtain functors
\[ \cat(d^*)\colon \diag\Sq(\C) \rightleftarrows \C\colon \cat(m^*)\]
with $\cat(d^*)\circ \cat(m^*)=\id$, and a natural transformation from $\cat(m^*)\circ \cat(d^*)$ to the identity. It remains to check that the latter natural transformation is an equivalence at each object of $\diag \Sq(\C)$. But this is clear from the definition since the natural transformation $H$ is stationary for $n=0$.
% , i.e. each object of $\C$. By naturality in $\C$, it suffices to consider the case $\C=[0]$, in which case $\diag \Sq[0]=[0]$ is a groupoid, so the claim follows.
\end{proof}

\begin{remark} \label{Segal-rem1}
Proposition \ref{prop:basic_adjunction} remains true when $\Cat$ is replaced by $\Seg$, that is, the functor $\Sq \colon \Seg \to \DSeg$ is fully faithful and admits a left adjoint induced by $\delta$. The proof is essentially the same in this case, too; the last argument of the proof is also valid for Segal spaces by \cite[Proposition 2.21, Remark 2.22]{Ara}. 
\end{remark}

\begin{remark}
%Proposition \ref{prop:basic_adjunction} makes no claim to originality. In fact, it 
We remark that a stronger fact is known, namely, that the extension of $\Sq$ to $(\infty,2)$-categories is also fully faithful \cite{Abe}, answering a conjecture of \cite{GR}. 
\end{remark}

\begin{remark} 
The essential image of $\Sq$, as part of a Bousfield localization, is given by all double Segal spaces $\D$ which are local for the class of $\diag$-equivalences, i.e., the maps that are inverted by $\diag$, but it is also possible to identify smaller generating subsets of $\diag$-equivalences in various ways. For example, consider the set of $\diag$-equivalences:
$$S_1 = \{ [m] \otimes [n] \to \Sq([m] \times [n]) \ | \ m, n \geq 0\},$$ 
given by unit maps of the adjunction $(\diag, \Sq)$, and 
$$S_2 = \{[m] \times [n]_{|v=\id} \to \Sq([m] \times [n]) \ | \ m,n \geq 0\},$$ 
given also by unit maps of the adjunction $(\diag, \Sq)$, where $[m] \times [n]_{| v=\id}$ denotes the horizontal category $[m] \times [n]$ with identities vertically. Then a bicomplete double Segal space $\D$ is in the essential image of $\Sq$ if and only if it is $(S_1 \cup S_2)$-local. Indeed, in this case, the resulting equivalences 
\[\Hom_{\DSeg}([\bullet] \otimes [\ast], \D) \leftarrow \Hom_{\DSeg}(\Sq([\bullet]\times [\ast]), \D) \rightarrow  \Hom_{\DSeg}([\bullet] \times [\ast]_{|v=\id}, \D)=\Hom_{\Cat}([\bullet]\times [\ast], \D_h)\]
identify $\D$ with $\Sq(\D_h)$, where $\D_h$ denotes the horizontal subcategory of $\D$. See also \cite{LR} for a study and a universal property of maps such as in $S_2$ for general $(\infty,2)$-categories.  
\end{remark}

Despite its foundational importance, the diagonal category is difficult to access in general directly from its definition; the technical reason is that the simplicial space $\delta(X)$ rarely is a Segal space, even if $X$ is a double category (``squares do not compose diagonally''). Our main result can be viewed as providing a description of $\diag(X)$ in terms of another construction that is more familiar and often more accessible. Namely, in view of Proposition \ref{prop:basic_adjunction}, our main result may be reformulated as stating a natural equivalence between $\diag(X)$ and the span category $\Span(X\pop2)$ to be defined below.

\subsection{The span category} 
The usual span category is the categorical realization of the $\Q$-construction \cite{Qui, Bar}, \cite{HHLN}, which is defined in turn in terms of the twisted arrow category. Our double-categorical variant is defined by replacing the twisted arrow category by a double-categorical variant thereof. 

For any $\C \in \Cat$, the \defi{twisted arrow double category} $\TwarD(\C)$ is the double category defined by 
\[\TwarD(\C) \colon \Delta\op \times \Delta\op \to \An, \  \TwarD(\C)_{p,q} := \Hom_{\Cat}([p]\ast[q]\op, \C);\]
the Segal condition in either direction follows from the fact that the join $\ast\colon \Cat\times \Cat\to \Cat$ preserves colimits in each variable separately. Evidently, the construction $\TwarD$ defines a functor $\Cat \to \DSeg$. 

We define the $\Q$-construction functor $\Q \colon \ssAn \to \sAn$ by setting for any $X \in \ssAn$
\[\Q(X) \colon \Delta\op \to \An, \ \Q(X)_n := \Hom_{\ssAn}(\TwarD[n], X),\]
and then define the \defi{span category} of $X \in \ssAn$ to be the category associated to the $\Q$-construction, i.e., it is given by the composite functor
$$\Span \colon \ssAn \xrightarrow{\Q} \sAn \xrightarrow{\mathrm{cat}} \Cat.$$
$\Q(\D)$ turns out to be a Segal space for a double category $\D$ more frequently than $\delta(\D)$ (see Corollary \ref{Segal_condition_for_S}), which makes the span category $\Span(\D)$ more accessible in general.

\begin{remark}\label{rem:diag_tward}
The diagonal category of $\TwarD(\C)$ is the usual twisted arrow category of $\C$; in fact, this is already true for the diagonal simplicial space $\delta(\TwarD(\C))$, which is well known to be a category ( = complete Segal space) in this case. 
\end{remark}

\begin{remark}\label{rem:usual_Q}
These definitions of the span category and of the $\Q$-construction generalize the usual ones, which take (single) categories as input. Indeed, for such $\C\in \Cat$, we have a natural equivalence of simplicial spaces
\[\Q(\Sq(\C))\simeq \Q(\C)\]
between the $\Q$-construction of the double category $\Sq(\C)$ and the usual $\Q$-construction of $\C$ (considered as a simplicial space, given by the degreewise groupoid core of a simplicial category that is also often denoted by $\Q(\C)$). This equivalence is obtained by means of the adjunction $(\delta, \Sq)$ and Remark \ref{rem:diag_tward}. It follows that our $\Span(\Sq(\C))$ is identified with the usual span category of $\C$. 

Moreover, for any (adequate) triple $(\C,\C_\dagger, \C^\dagger)$ in the sense of Barwick \cite{Bar3}, we may consider the bisimplicial subspace $\Sq(\C,\C_\dagger, \C^\dagger, \mathrm{cart})\subset \Sq(\C)$ which contains only the $(p,q)$-simplices $X\colon [p]\otimes [q]\to \Sq(\C)$  that send each horizontal morphism $\{i\}\times \{k\leq l\}$ to  $\C_\dagger$, each vertical morphism $\{i\leq j\}\times \{k\}$ to $\C^\dagger$, and each square $\{i < j\}\times \{k < l\}$ to a cartesian square. Then the conditions for $\TwarD[n] \to \Sq(\C)$ to lie in $\Sq(\C, \C_\dagger, \C^\dagger, \cart)$ match exactly the conditions for the associated functor $\Twar[n]\to \C$ to lie in the usual $\Q$-construction of $(\C, \C_\dagger, \C^\dagger)$, so that we obtain a natural equivalence of simplicial spaces
\[ \Q(\Sq(\C,\C_\dagger, \C^\dagger, \mathrm{cart}))\simeq \Q(\C, \C_\dagger, \C^\dagger)\]
with the $\Q$-construction (or effective Burnside $\infty$-category) of this (adequate) triple (see \cite[5.8-5.10]{Bar3}). Again, we obtain an identification of the corresponding span categories by passing to the associated categories.
\end{remark}

\subsection{The step category} 
The step category of a bisimplicial space $X$ is the category $\Span(X\pop2)$ where $(-)\pop2$ indicates passing to the opposite in the second simplicial direction. Let us make this more explicit. For any $\C \in \Cat$, the \defi{arrow double category} $\ArD(\C)$ of $\C$ is given by the double Segal space
\[\ArD(\C) \colon \Delta\op \times \Delta\op \to \An, \ \ArD(\C)_{p,q} :=  \Hom_{\Cat}([p]\ast[q], \C).\]
In other words, we have an identification $\ArD(\C) = \TwarD(\C\pop2)$.
By analogy with the $\Q$-construction, this construction also yields a functor $\oS\colon \ssAn \to \sAn$ defined for any $X \in \ssAn$ by 
\[\oS(X) \colon \Delta\op \to \An, \ \oS(X)_n := \Hom_{\ssAn}(\ArD[n], X)\]
which we refer to as \emph{$\oS$-construction}, because it is an unpointed variant of the $\oS$-construction \cite{Wal} extended to the double-categorical context (see also Section \ref{double_vs_single_span}). 
Note that for any $X \in \ssAn$, we have a canonical identification:
\begin{equation} \label{S-vs-Q} \tag{$\Q$ vs. $\oS$}
\oS(X)\simeq\Q(X\pop2).
\end{equation}
We then define the \defi{step category} of $X$ as the category associated to the $\oS$-construction, i.e., it is given by the composite functor
$$\Stair \colon \ssAn \xrightarrow{\oS} \sAn \xrightarrow{\mathrm{cat}} \Cat$$
and it is related to the span category by $\Stair(X) \simeq \Span(X\pop2)$. Informally, the objects of $\Stair(\D)$, say for $\D\in \DSeg$, are just the objects of $\D$, and each ``step'' in $\D$
\[
\begin{tikzcd}
     c \ar[r] & d \ar[d]\\ & e
\end{tikzcd}
\]
consisting of a horizontal followed by a vertical morphism of $\D$, gives a morphism in $\Stair(\D)$. (In complete analogy to the span category, a general morphism in $\Stair(\D)$ is a composite of such steps, where the length can be reduced to $1$ whenever $\Stair(\D)$ is 
a Segal space.)

\begin{remark}
Similarly to $\TwarD(\C)$, it is also true that the diagonal category of $\ArD(\C)$ is the usual arrow category of $\C$, but this is less evident (because the simplicial space $\delta(\ArD(\C))$ usually fails to be a Segal space) and a consequence of Theorem \ref{S-constr} and Lemma \ref{lem:diagonal_of_ArD} below. 
\end{remark}

\begin{remark}\label{rem:juran}
This version of the $\oS$-construction has also been considered by Juran \cite{Juran} (under the name Cnr) in the situation where it yields a complete Segal space -- he has also obtained Corollary \ref{Segal_condition_for_S}(1) below, which describes when exactly this happens. In this context, he proved that the Step category (in our terminology) is canonically equipped with an orthogonal factorization system given by the images of the horizontal and the vertical morphisms, and that it defines an equivalence between a full subcategory of $\DCat$ and the category of orthogonal factorization systems.
\end{remark}

\subsection{Reformulation of Theorem \ref{span-constr}} 
The definition of the step category $\Stair(\D) = \cat( \oS(\D))$ was meant to emphasize the analogy with the span category $\Span(\D)= \cat(\Q(\D))$ and unify the $\Q$- and $\oS$-constructions in the double-categorical context, but our main result actually amounts to saying that the step category is (again) just the diagonal category. 

\smallskip

There is a natural transformation $w \colon \delta \to \oS$ that is induced by precomposition with the canonical natural inclusion 
$$\ArD[n] \to [n] \otimes [n], \quad (i \leq j) \mapsto (i, j).$$

Using \eqref{S-vs-Q}, the following Theorem shows the first part of Theorem \ref{span-constr} in a more general context where $\DSeg$ is replaced by $\ssAn$. It will be shown in Section \ref{proof-main-thm}. (The second part of Theorem \ref{span-constr} then follows from Proposition \ref{prop:basic_adjunction}.)

\begin{theorem}\label{S-constr}
For each bisimplicial space $X$, the natural map $w_X \colon \delta(X) \to \oS(X)$ becomes an equivalence after categorical realization.
\end{theorem}

We remark that the inverse of $\cat(w_X) \colon \diag(X) \xrightarrow{\simeq} \Stair(X)$ does not arise from a simplicial map $\oS(X) \to \delta(X)$ before categorical realization. 

\smallskip

When applied to a bisimplicial set $X$ thought of as a discrete bisimplicial space, this theorem recovers \cite[Theorem 1.4.26]{LZ}. In this case, $w_X$ is a map of simplicial sets (= discrete simplicial spaces), and the composite functor $\sSet\subset \sAn\to \Cat$ is the localization at the Joyal equivalences (as follows from the work of Joyal--Tierney \cite[Section 4]{JT}, see also \cite[Remark 2.4]{HS}). Moreover, the special case of our result for $X = \Sq(\C, \C_h, \C_v, \cart) \subset \Sq(\C)$, i.e., a double subcategory of $\Sq(\C)$ whose squares are restricted to pullbacks, recovers the corresponding universal property stated in \cite[III.7, 2.1.3]{GR}; the precise relation to \cite[III.7, 2.1.3]{GR}, which is formulated in terms of $(\infty,2)$-categories and uses a differerent construction of span category, seems unclear in general.

\begin{remark} \label{Seg-rem3}
The proof of Theorem \ref{S-constr} also shows that $w_X$ becomes an equivalence already after passing to the associated Segal spaces. In combination with Remark~\ref{Segal-rem1}, this shows that the span-squares adjunction extends to an adjunction:
\begin{equation} \tag{$\dashv$}
L_{\Seg} \circ \Q \colon \ssAn \rightleftarrows \Seg \colon \Sq(-)\pop2
\end{equation}
where $L_{\Seg} \colon \sAn \to \Seg$ denotes the localization functor/Segal space reflection.
\end{remark}

\section{Analysis of the arrow double category}\label{sec:arrow_double}

\subsection{A colimit decomposition of $\ArD[n]$.} 

The proof of Theorem \ref{S-constr} relies essentially on an analysis of the arrow double category of $[n]$. Note that $\ArD[n]\in \DSeg \subset \ssAn$ is a bisimplicial set, i.e., bi-levelwise discrete. We denote by $\SpineD[n] \subseteq \ArD[n]$ the bisimplicial subset which is spanned by all objects; $(0,1)$-simplices of the form
$$(i \leq j) \to (i \leq j+1);$$
$(1,0)$-simplices of the form 
$$(i \leq j) \to (i+1 \leq j);$$
and $(1,1)$-simplices of the form
\[
 \begin{tikzcd}
  (i \leq j) \ar[r] \ar[d] & (i \leq j+1) \ar[d]\\
  (i+1 \leq j) \ar[r] & (i+1 \leq j+1);
 \end{tikzcd}
\]
$\SpineD[n]$ can be depicted in the following way: 
\begin{equation}\label{eq:ArD} \tag{DSpine}
 \begin{tikzcd}
  (0 \leq 0) \ar[r] & (0\leq 1) \ar[d] \ar[r] & (0\leq 2) \ar[d] \ar[r] & \dots \ar[r]&  (0  \leq n) \ar[d]\\
    & (1\leq 1) \ar[r] & (1\leq 2) \ar[d] \ar[r]   & \dots \ar[r] & (1\leq n) \ar[d]\\
    &&\ddots & \dots & \vdots \ar[d]\\
    &&& & (n\leq n).
 \end{tikzcd}
\end{equation}
We note that this picture is commonly also used to depict the category $\Ar[n]$. However, we distinguish here between the horizontal and vertical edges, and squares are understood in a bisimplicial sense. Moreover, there are no hidden composite morphisms 
or squares in $\SpineD[n]$.

\begin{lemma}\label{lem:decomposition_ArD}
The inclusion $\SpineD[n] \hookrightarrow \ArD[n]$ becomes an equivalence of double Segal spaces after double-categorical realization. 
\end{lemma}

\begin{proof}
We proceed by induction on $n$. The claim is obvious for $n=0$. Let $n>0$ and write the bisimplicial set $\SpineD[n]$ as the union of two 
subobjects $A_1$ and $A_2$. $A_1$ is spanned by $(i \leq j)$ with $j \leq n-1$, and $A_2$ is spanned by $(i \leq j)$ with $j \geq n-1$, and their intersection $A_0$ is spanned by $(i \leq j)$ with $j=n-1$. This gives a pushout decomposition of $\SpineD[n]$ in $\ssSet$, and hence also in $\ssAn$. 

By induction, the inclusion of bisimplicial sets
$$A_1 \simeq \SpineD[n-1] \to \ArD[n-1]$$ 
becomes an equivalence after double-categorical realization. Hence, given that double-categorical realization preserves colimits, it  suffices to show that the canonical inclusion
\begin{equation} \label{canonical-incl} \tag{*}
\ArD[n-1] \cup_{A_0} A_2 \to \ArD[n]
\end{equation}
is an equivalence after double-categorical realization. This claim follows from the stronger assertion that the map \eqref{canonical-incl} is an equivalence of simplicial objects in Segal spaces, regarded levelwise in each vertical degree, i.e., we claim that for each $[p]\in \Delta$, the induced inclusion of simplicial (discrete) spaces
\begin{equation} \label{canonical-incl2} \tag{**}
\ArD[n-1]_{p,\bullet} \cup_{(A_0)_{p,\bullet}} (A_2)_{p,\bullet} \to (\ArD[n])_{p,\bullet}
\end{equation}
becomes an equivalence after Segal space reflection. 

To see this, we observe that the target simplicial space of \eqref{canonical-incl2} decomposes canonically into a disjoint union indexed by the maps $\alpha\colon [p] \to [n]$, according to the value of each $(p,q)$-bisimplex $[p]\ast [q]\to [n]$ of $\ArD[n]$ on the first join summand (which remains constant under the simplicial operations in the second direction). The summand corresponding to $\alpha \colon [p] \to [n]$ is (the nerve of) the category
$$\{\alpha(p)\leq \dots \leq n\}=:\alpha(p)^+$$ 
since the restriction of a $(p, q)$-bisimplex $[p] \ast [q] \to [n]$ on the second join summand can take arbitrary values $\geq \alpha(p)$. This decomposition induces a decomposition of all terms that arise in the map \eqref{canonical-incl2}, which is compatible with the map and with the pushout decomposition of the domain. Explicitly, the restriction of the map \eqref{canonical-incl2} on each summand $\alpha \colon [p] \to [n]$ becomes isomorphic to 
\[(\alpha(p)^+ \cap [n-1]) \cup_{\alpha(p)^+\cap \{n-1\}} (\alpha(p)^+\cap \{n-1\leq n\}) \to \alpha(p)^+\]
which is indeed an equivalence after Segal space reflection (also for $\alpha(p)=n$).
\end{proof}

\begin{corollary}\label{Segal_condition_for_S} Let $\D \in \DSeg$.
\begin{enumerate}
    \item[(1)] The simplicial space $\oS(\D)$ is a Segal space if and only if each diagram
\[\xymatrix{
 d_{01}   \ar[d]\\
 d_{11}  \ar[r]  & d_{12} 
}\]
in $\D$ (consisting of a morphism in the first direction, followed morphism in the second direction) extends uniquely to a square (= $(1,1)$-bisimplex) in $\D$. More precisely, the corresponding boundary map 
$$\D_{11}\to \D_{10} {}^t\times^s_{\D_{00}} \D_{01}$$ 
is an equivalence of spaces.
\item[(2)] The simplicial space $\Q(\D)$ is a Segal space if and only if each diagram
\[\xymatrix{
d_{01} \ar[d] &\\
d_{11} & d_{12} \ar[l] \\
}\]
in $\D$ extends uniquely to a square in $\D$. More precisely, the corresponding boundary map 
$$\D_{11}\to \D_{10} {}^t\times^t_{\D_{00}} \D_{01}$$ 
is an equivalence of spaces.
\end{enumerate}
\end{corollary}

\begin{proof} 
We only need to prove (1); (2) is just the $(-)\pop2$-dual statement using $\Q(\D)=\oS(\D\pop2)$. By definition, $\oS(\D)$ is a Segal space if the restriction map 
$$\Hom_{\ssAn}(\ArD[n], \D) \to \Hom_{\ssAn}(\ArD(\{0\leq 1\})\cup_{\ArD(\{1\})} \dots \cup_{\ArD(\{n-1\})} \ArD(\{n-1\leq n\}), \D)$$
along the canonical inclusion is an equivalence for $n \geq 2$. 
By Lemma \ref{lem:decomposition_ArD}, we may replace $\ArD[n]$ by $\SpineD[n]$. 
Then note that each inclusion 
$$
\ArD(\{0\leq 1\})\cup_{\ArD(\{1\})} \ArD(\{1\leq 2\})\cup_{\ArD(\{2\})}\dots \cup_{\ArD(\{n-1\})} \ArD(\{n-1\leq n\})\to \SpineD[n]
$$
is obtained by iterative pushouts along 
\[
 \begin{tikzcd}
  \bullet \ar[d] \\ \bullet \ar[r] & \bullet
 \end{tikzcd}
 \quad \to \quad 
 \begin{tikzcd}
  \bullet \ar[r] \ar[d] & \bullet \ar[d] \\
  \bullet \ar[r] & \bullet
 \end{tikzcd}
\]
As a consequence, the Segal condition on $\oS(\D)$ reduces to the unique extension property of $\D$. This proves the ``if'' statement; for the converse it suffices to note that the boundary map in question is a retract of the Segal map in degree $2$. The retraction is given by forgetting the value of the diagrams at $(0\leq 0)\to (0\leq 1)$ and $(1\leq 2)\to (2\leq 2)$, and the retract-inclusion extends the diagrams using the identities.
\end{proof}

\subsection{The diagonal category of $\ArD(\C)$.} 
We can identify the $\Stair$-category (hence, \emph{a posteriori} also the diagonal category) of $\ArD(\C)$ with the usual arrow category of $\C$. In fact, already the $\oS$-construction is a category in this case:

\begin{lemma}\label{lem:diagonal_of_ArD}
For any $\C \in \Cat$, there is a natural equivalence of simplicial spaces $\oS(\ArD(\C)) \simeq \Ar(\C)$. 
\end{lemma}
\begin{proof} First note that $\oS(\ArD(\C))$ is a Segal space by Corollary \ref{Segal_condition_for_S}. Indeed, $\ArD(\C)$ satisfies the hypothesis of Corollary \ref{Segal_condition_for_S} by the following pushout square in $\Cat$
\[
 \begin{tikzcd}
  \{1\} \ast \{0\} \ar[r] \ar[d] & \{1\} \ast \{0\leq 1\} \ar[d]\\
  \{0\leq 1\} \ast \{0\} \ar[r] & \{0\leq 1\} \ast \{0 \leq 1\}. 
 \end{tikzcd}
\]
We define the desired natural equivalence $\theta_{\C} \colon \Ar(\C) \to \oS(\ArD(\C))$, first on $0$- and $1$-simplices: it is the identity on $0$-simplices, and on $1$-simplices it rearranges a commutative square in $\C$ as the composite of a horizontal followed by a vertical arrow of $\ArD(\C)$:
\[
 \begin{tikzcd}
  a \ar[r] \ar[d] \ar[r]  & b \ar[d]\\
  c \ar[r] & d
 \end{tikzcd}
 \quad \mapsto\quad 
 \begin{tikzcd}
  (a\to b) \ar[r] &  (a\to  d) \ar[d] \\
  ~ & (c\to d).
 \end{tikzcd}
\]
This is an equivalent codification of the same data as can be seen from the following pushout square in $\Cat$
\[
 \begin{tikzcd}
   \{0\}\ast \{1\} \ar[r] \ar[d] & \{0\} \ast \{0\leq 1\} \ar[d] \\
   \{0\leq 1\} \ast \{1\} \ar[r] & {[1]\times [1]}.
 \end{tikzcd}
\]
More generally, the map $\theta_{\C}$ on $n$-simplices is defined by
\[\Hom_{\Cat}([n]\times [1], \C) \xrightarrow{\ArD} \Hom_{\ssAn}(\ArD([n]\times [1]), \ArD(\C) ) \to \Hom_{\ssAn} (\ArD[n], \ArD(\C))\]
where the second map restricts along the map (natural in $[n]$) 
\[(\id, \const_{0\leq 1}) \colon \ArD[n]\to \ArD[n]\times \ArD[1] \simeq \ArD([n]\times [1]).\]
Then $\theta_{\C}$ is an equivalence of Segal spaces because it is an equivalence on $0$- and $1$-simplices. 
\end{proof}

%%%%%%%%%%%%%%%%%%%%%%%%%%%%%%%%%%%%%%%%%%%%%%%%%%%%%%%
%%%%%%%%%%%%%%%%%%%%%%%%%%%%%%%%%%%%%%%%%%%%%%%%%%%%%%%%%

\section{Proof of Theorem \ref{S-constr}} \label{proof-main-thm}

We first prove the following special case of Theorem \ref{S-constr}. 

\begin{lemma}\label{lem:diag_ArD}
The functor $\cat(w_{\ArD[n]}) \colon \diag(\ArD[n]) \to \Stair(\ArD[n])\;(\simeq \Ar[n])$ is an equivalence.
\end{lemma}

\begin{proof}
By Lemma \ref{lem:diagonal_of_ArD}, the target is equivalent to $\Ar[n]$, so we can consider the functor in question as a functor $\diag(\ArD[n])\to \Ar[n]$, which is still the identity on objects (and is specified by this property because the target is a poset).

We show next that the source of this functor is a poset, by showing it to be a retract of $[n]\times [n]$. A functor $I\colon \diag (\ArD[n])\to [n]\times [n]$ is obtained from the embedding $\ArD[n]\to [n]\otimes [n]$ (which sends $i\leq j$ to $(i,j)$) by passing to diagonals. In order to find a retraction to $I$, we view $[n]$ as the categorical realization of $\Spine[n]=\{0\leq 1\}\cup_{\{1\}} \dots \cup_{\{n-1\}} \{n-1\leq n\}$, and consider the map $R\colon \Spine[n]\times \Spine[n]\to \diag(\ArD[n])$ defined as follows:
\begin{equation}\label{eq:square-diagram}
\begin{tikzcd}
    (0\leq 0) \ar[r] \ar[d] & (0\leq 1) \ar[r] \ar[d] & \dots \ar[r] & (0\leq n) \ar[d]\\
    (0\leq 0) \ar[r] \ar[d] & (1\leq 1) \ar[r] \ar[d] & \dots \ar[r] & (1\leq n) \ar[d]\\
    \vdots \ar[d] & \vdots\ar[d] && \vdots\ar[d] \\
    (0\leq 0) \ar[r] & (1\leq 1) \ar[r] & \dots \ar[r] & (n\leq n).
\end{tikzcd}
\end{equation}
Note that this is \emph{not} well defined in $\delta(\ArD[n])$ because of the part of the diagram which is below the diagonal. In detail, the map $R$ can be constructed as follows. First note that the inclusion $\ArD[n] \subseteq  [n]\otimes [n]$ restricts to an inclusion $\delta(\SpineD[n])\subseteq \Spine[n]\times \Spine[n]$ of simplicial sets which contains exactly the simplices that have no vertex of the form $(i,j)$ with $i>j$ (i.e., no vertex situated lower-left to the diagonal); visually, this is simply the obvious inclusion of \eqref{eq:ArD} into \eqref{eq:square-diagram}.  The functor $R$ is defined on $\delta(\SpineD[n])$ to be the inclusion, and in order to find the desired extension to the remaining squares of $\Spine[n] \times \Spine[n]$, we need to find successively extensions along inclusions of the kind
\[
\begin{tikzcd}
    \bullet \ar[r] & \bullet \ar[d]\\ & \bullet 
\end{tikzcd}
\quad
\subset 
\quad 
\begin{tikzcd}
 \bullet \ar[r] \ar[d] & \bullet \ar[d] \\ \bullet \ar[r] & \bullet
\end{tikzcd}
\]
But this can clearly be done for any target category, because this inclusion itself admits a retraction after passing to the associated categories. This concludes the construction of $R$. 

To see that $R$ is a retraction, we recall from Lemma \ref{lem:decomposition_ArD} that the inclusion $\SpineD[n] \hookrightarrow \ArD[n]$ becomes an equivalence after double-categorical realization, so that the induced inclusion $\delta(\SpineD[n]) \hookrightarrow \delta(\ArD[n])$ becomes an equivalence after categorical realization. (Indeed $\delta\colon \ssAn\to \sAn$ sends double-categorical equivalences to categorical equivalences because $\Sq$ sends categories to double Segal spaces.) But on $\delta(\SpineD[n])$ the map $R$ was defined to be the identity.  

Thus, the functor $\diag(\ArD[n])\to \Ar[n]$ under consideration is indeed a map of posets which is a bijection on objects. It remains to show that it is full. For this, consider  $(i\leq j)\to (i'\leq j')$ in $\Ar[n]$, i.e., $i\leq i'$ and $j\leq j'$; then this also defines a morphism $(i,j)\to (i', j')$ in $[n]\times [n]$ which gives, under the retraction $[n] \times [n] \to \diag(\ArD[n])$, a desired preimage in $\diag(\ArD[n])$. 
\end{proof}

% \begin{remark} \label{Segal-rem2}
% The same proof shows that $w_{\ArD[n]}$ becomes an equivalence already after passing to the associated Segal spaces (in fact, these are already complete).
% \end{remark}

\begin{remark} 
Using Lemma \ref{lem:diag_ArD}, we find that the colimit decomposition of $\ArD[n]$ from Lemma \ref{lem:decomposition_ArD} induces a colimit decomposition of $\Ar[n]$ by passing to diagonals. Namely, the composite inclusion 
\[\Spine\Ar[n]:=\delta(\SpineD[n])\subset \delta(\ArD[n]) \subset \Ar[n]\]
is a categorical equivalence of simplicial spaces. Explicitly, $\Spine\Ar[n]$ is the subobject of the simplicial space (or set)  $\Ar[n]$ that is spanned by $1$-simplices of the form $(i \leq j) \to (i \leq j+1)$ and $(i \leq j) \to (i+1 \leq j)$, and $2$-simplices of the form
\[
  \begin{tikzcd}
   (i \leq j) \ar[r] \ar[d] & (i \leq j+1) \ar[d]\\
   (i+1 \leq j) \ar[r] & (i+1 \leq j+1).
  \end{tikzcd}
\]
It is also possible to first prove this result directly, by an inductive argument, and then to deduce Lemma \ref{lem:diag_ArD} from this and Lemma \ref{lem:decomposition_ArD}. 
\end{remark}

\smallskip

\begin{proof}[Proof of Theorem \ref{S-constr}]
We will more generally identify the spaces of natural transformations between functors $\ssAn \to \Cat \subset \sAn$:
$$\Hom(\Stair, \diag), \ \Hom(\diag, \diag), \ \text{ and } \Hom(\Stair, \Stair)$$
and show that they are all contractible. This would readily imply the existence of an (essentially unique) natural transformation $\Stair \to \diag$ which is an inverse to $\cat(w)$, as required. 

Since $\Cat \subset \sAn$ is a full reflective subcategory, the canonical natural transformation $\oS \to \Stair$ induces an identification of spaces of natural transformations of functors $\ssAn\to \sAn$
$$\Hom(\Stair, \diag) \simeq \Hom(\oS, \diag).$$
The space of natural transformations $\oS \to \diag$ (viewed as functors on $\Delta\op$) is identified as follows:
\[
\Hom(\oS, \diag) = \lim_{([m]\to [n])\in \Twar(\Delta\op)} \Hom(\oS_n, \diag_m) 
\simeq \lim_{([m]\to [n])\in \Twar(\Delta\op)} \diag_m(\ArD[n])\]
where we have used that $\oS_n$ is represented by $\ArD[n]$ and the Yoneda lemma. Moreover, since $X_m \simeq \Hom_{\sAn}([m], X)$ naturally in $X$ and $[m]\in \Delta$, we can rewrite this last expression  as
\[ 
\lim_{([n]\to [m])\in \Twar(\Delta)} \Hom([m], \diag(\ArD[n])) \simeq \Hom([-], \diag(\ArD[-])) \simeq \Hom([-], \Ar[-])
\]
where we have used Lemmas \ref{lem:diag_ArD} and \ref{lem:diagonal_of_ArD} in the last place; the latter $\Hom$ is the space of natural transformations between cosimplicial objects in categories (in fact, posets). 

Now we note that for any cosimplicial poset $P(-)$, natural transformations $\alpha\colon [-]\to P(-)$ are determined by the individual maps $\alpha_m\colon [m]\to P(m)$, which in turn are determined by their values $\alpha_m(i)$, and the latter are simply determined by the single value of $\alpha$ at $0\in [0]$. 
In addition, the resulting injective map from $\Hom([-], P(-))$ into (the underlying set of) $P(0)$ is a bijection, since any choice of 
$x\in P(0)$ extends to the natural transformation that sends $i\in [m]$ to $P(i\colon [0]\to [m])(x)$. In summary, we conclude that
\[\Hom(\Stair, \diag) \simeq \Hom([-], \Ar[-]) \simeq \Ar[0]\]
is indeed contractible. 
Entirely similar arguments also show canonical equivalences
\[\Hom(\diag, \diag) \simeq \Hom(\delta, \diag) \simeq \Hom([-], \diag([-]\otimes [-]) \simeq \Hom([-], [-]\times [-]) \simeq [0]\times [0]\]
and 
\[\Hom(\Stair, \Stair) \simeq \Hom(\oS, \Stair) \simeq \Hom([-], \Stair(\ArD[-]) \simeq \Hom([-], \Ar[-]) \simeq \Ar[0],\]
where we have used Lemma \ref{lem:diagonal_of_ArD} again. So all these mapping spaces of natural transformations are again contractible as claimed.
(It would not be difficult to write down explicitly the transformation $\oS\to \diag$ using the Yoneda lemma --pointwise--, and then demonstrate the commutativity of the triangles below more explicitly
\begin{equation*}\label{eq:split_in_square}
\begin{tikzcd}
\delta \ar[r, "w"] \ar[d] & \oS \ar[d] \ar[ld, ""']\\
\diag \ar[r, "\cat(w)"'] & \Stair
 \end{tikzcd}
\end{equation*}
using, again, the Yoneda lemma.)
\end{proof}

\begin{proof}[Proof of Theorem \ref{span-constr}]
Combine Theorem \ref{S-constr} with Proposition \ref{prop:basic_adjunction}.
\end{proof}

%%%%%%%%%%%%%%%%%%%%%%%%%%%%%%%%%%%%%%%%%%%%%%%%%%%%%%%%%%%%%%%%
%%%%%%%%%%%%%%%%%%%%%%%%%%%%%%%%%%%%%%%%%%%%%%%%%%%%%%%%%%%%%%%%

\section{Relative categories, double categories, and localizations}\label{localizations}

Recall that a \defi{relative category} is an ($\infty$-)category $\C$ together with a subcategory $W$ containing the groupoid core $\C^\simeq$ of $\C$; these form the objects of the category $\RelCat$ of relative categories and relative functors, i.e.~functors that respect the subcategories. To a relative category $(\C, W)$ we can associate the double category 
\[\Sq(\C, W)\subset \Sq(\C)\]
which is the double subcategory whose horizontal (i.e., second-direction) morphisms are restricted to $W$. On the other hand, for any double category $\D$ we may consider its diagonal category $\diag(\D)$ as a relative category where we define the subcategory as the one generated by the image of the horizontal subcategory $\D_h=\diag(\D_h\vert_{v=\id})$. The following extends Proposition \ref{prop:basic_adjunction} to the setting of relative categories.

\begin{proposition}\label{prop:basic_relative_adjunction}
The resulting functors are part of an adjunction 
\[\diag\colon \DSeg \rightleftarrows \RelCat \colon \Sq\]
with $\diag$ the left adjoint. The counit transformation is an equivalence (so that $\Sq$ identifies $\RelCat$ with a Bousfield localization of $\DSeg$).
\end{proposition}

\begin{proof}
For the adjunction from Proposition \ref{prop:basic_adjunction}, one verifies that a double functor $\D\to \Sq(\C)$ sends the horizontal part $\D_h$ into $W$ if and only if its adjoint $\diag(\D)\to \C$, when precomposed with $\D_h\to \diag(\D)$, takes values in $W$. This establishes the first statement. The second statement amounts to saying that the map $\diag \Sq(\C,W)\to \C$ is an equivalence. This holds for $W=\C$ by Proposition \ref{prop:basic_adjunction}. Thus, it suffices to see that the inclusion $\Sq(\C, W)\hookrightarrow \Sq(\C)$ is an equivalence on diagonal categories. To see this, we note that the simplicial map from the proof of Proposition \ref{prop:basic_adjunction}:
$$\delta(\Sq(\C)) \times [1] \to \delta(\Sq(\C))$$ 
extends to a map 
$$\oS (\Sq(\C)) \times [1] \to \oS(\Sq(\C))$$ 
by replacing $[n] \times [n]=\diag([n]\otimes [n])$ in the formula with $\Ar[n] \simeq \diag(\ArD[n])$ (Lemma \ref{lem:diag_ArD}). This new map has the properties that: (i) it is a simplicial homotopy to the identity from a map that takes values in $\oS(\Sq(\C, W))\subset \oS(\Sq(\C))$, (ii) it restricts to a simplicial homotopy on this subspace, and (iii) it is still stationary on the $0$-simplices. The claim then follows after categorical realization, and using  $\cat\circ \oS\simeq \diag$ by Theorem \ref{S-constr}.
\end{proof}

%A corollary of this result is the relative nerve theorem: In our language, the relative nerve $\Nrel(\C, W)$ is the simplicial space $\vert \Sq(\C, W)\vert_2$ (geometric realization in the second, i.e., horizontal direction).

Proposition \ref{prop:basic_relative_adjunction} can be viewed as a strengthening of the relative nerve theorem as we explain next. Recall that the relative nerve $\Nrel(\C, W)$ of a relative category $(\C, W)$ is the simplicial space $\vert \Sq(\C, W)\vert_2$, given by geometric realization in the second, i.e., horizontal, direction.

\begin{theorem}\label{thm:loc_as_span}
For every relative category $(\C, W)$, we have natural equivalences 
\[\cat \Nrel(\C, W)\simeq \C[W^{-1}] \simeq \Span(\Sq(\C, W)).\]
\end{theorem}

The left-hand side equivalence is known as the relative nerve theorem and is due in this form to Mazel-Gee \cite{MG} (see also \cite{AC} for a recent short proof and a survey of its history which originates from the work of Rezk). The right-hand side equivalence seems to be new and gives, through the necklace formula, a description of mapping spaces in $\C[W^{-1}]$ in terms of higher spans (with no condition that the corresponding squares are cartesian). 
\begin{proof}
By Proposition \ref{prop:basic_relative_adjunction}, relative functors $(\C, W)\to (\D, \D^\simeq)$ correspond to double functors $\Sq(\C, W)\to \Sq(\D, \D^\simeq)$, where we note that the target is constantly $\D$ in the second direction. Thus, we obtain a unique factorization first over $\vert - \vert_2$, which is exactly $\Nrel(\C, W)$, and then over its categorical realization. This establishes the left-hand side equivalence. 

To see the universal property of localization for the right-hand term, we use the equivalence $\Sq(\D, \D^\simeq)\simeq \Sq(\D, \D^\simeq)\pop2$, resulting from the constancy in the second direction, as to further identify double functors $\Sq(\C, W)\to \Sq(\D, \D^\simeq)$ with relative functors $\diag(\Sq(\C, W)\pop2) \to (\D, \D^\simeq)$ (using again Proposition \ref{prop:basic_relative_adjunction}). By Theorem \ref{S-constr}, the source category is $\Span(\Sq(\C, W))$  equipped with the subcategory generated by morphisms of the form $(c\xleftarrow{w} d \xrightarrow{\id} d)$ for $w\in W$. It remains to show that these morphisms are invertible in $\Span(\Sq(\C, W))$. Indeed, we claim that $(d\xleftarrow\id d \xrightarrow w c)$ is an inverse: the left-hand and middle diagrams 
    \[
     \begin{tikzcd}[sep=tiny]
        && d \ar[rd] \ar[ld]\\
        & d \ar[rd] \ar[ld] && d \ar[ld] \ar[rd]\\
        d && c && d
    \end{tikzcd}
    \qquad 
         \begin{tikzcd}[sep=tiny]
        && d \ar[rd] \ar[ld]\\
        & d \ar[rd] \ar[ld] && d \ar[ld] \ar[rd]\\
        c && d && c
    \end{tikzcd}
    \qquad 
     \begin{tikzcd}[sep=tiny]
        && d \ar[rd] \ar[ld]\\
        & c \ar[rd] \ar[ld] && c \ar[ld] \ar[rd]\\
        c && c && c
    \end{tikzcd}
    \]
    viewed as $2$-simplices in $\Q \Sq(\C,W)$, show that the two possible composites are $\id_d$ and $(c\leftarrow d \rightarrow c)$ respectively, whereas the right-hand side $2$-simplex shows that the latter morphism is indeed $\id_c\circ \id_c=\id_c$.
\end{proof}

\section{Comparison of models for algebraic $K$-theory}\label{double_vs_single_span}

\subsection{Comparison with the $\Q$-construction of exact categories} We now compare our version of the $\Q$-construction and of the span category with the usual $\Q$-construction of an exact category \cite{Qui, Bar}. 

For an exact ($\infty$-)category $(\C, \C_\dagger, \C^\dagger)$ in the sense of \cite{Bar, Bar2}, its $\Q$-construction is defined as the $\Q$-construction of this adequate triple. Thus, using Remark \ref{rem:usual_Q}, we conclude that this agrees with the $\Q$-construction of the double category 
\[\Sq(\C, \C_\dagger, \C^\dagger, \cart) \subset \Sq(\C).\]
As a consequence of Theorem \ref{S-constr}, we deduce the following agreement of models for the algebraic $K$-theory of exact categories: via the $\Q$-construction and via
an $\infty$-categorical generalization of Thomason's construction \cite[end of Section 1.3]{Wal}/squares $K$-theory (see \cite{CKMZ}). 

\begin{corollary}
For any exact category $(\C, \C_\dagger, \C^\dagger)$, we have 
\[\vert \Q(\C, \C_\dagger, \C^\dagger) \vert \simeq \vert \Sq(\C, \C_\dagger, \C^\dagger, \cart)\vert.\]
\end{corollary}

\subsection{Comparison with the $\oS$-construction of Waldhausen categories} We now compare our version of the $\oS$-construction and of the step category with usual $\oS$-construction of (unpointed) Waldhausen categories \cite{Wal, Bar4}. 

Let $(\C, \co(\C))$ be an unpointed Waldhausen category, by which we mean an ($\infty$-)category $\C$ equipped with a wide subcategory $\co(\C)$ that is closed under cobase change in $\C$. We consider the double Segal subspace 
\[\Sq(\C, \co(\C), \C, \cocart)\subset \Sq(\C)\]
which contains only the $(p,q)$-simplices $X\colon [p]\otimes [q]\to \Sq(\C)$ that send each horizontal morphism to  $\co(\C)$, and each square $\{i < j\}\times \{k < l\}$ to a cocartesian square. This may be viewed as an $\infty$-categorical generalization of Thomason's construction (see \cite[end of Section 1.3]{Wal}). 

On the other hand, we also define an \emph{unpointed} $\oS$-construction $\oSun(\C, \co(\C))$ as follows. $\oSun(\C, \co(\C))$ is a simplicial space whose $n$-simplices are given by the space of functors $A\colon \Ar[n]\to \C$ required to satisfy the usual conditions of the $\oS$-construction \cite[Section 1.3]{Wal}, except for the condition $A_{i\leq i}=*$. If $\C$ has a zero object $\ast$, we denote by $\oS(\C, \co(\C)) \subset \oSun(\C, \co(\C))$ the simplicial subspace of functors $A \colon \Ar[n] \to \C$, $n \geq 0$, such that, in addition, $A_{i \leq i} = \ast$ holds. % -- this is an analogue of the $\os$-construction \cite{Wal}.

\begin{proposition}\label{S_is_S}
The $\oS$-construction of $\Sq(\C, \co(\C), \C, \cocart)$ agrees with the unpointed $\oS$-construction  $\oSun(\C, \co(\C))$. 
\end{proposition}
\begin{proof} 
This is the same as the argument from Remark \ref{rem:usual_Q}, but using the identification $\diag (\ArD(\C)) \simeq \Ar(\C)$ of Lemma \ref{lem:diagonal_of_ArD} instead of the (trivial) identification $\diag(\TwarD(\C))\simeq \Twar(\C)$.
\end{proof}

As a consequence of Theorem \ref{S-constr}, we deduce the following agreement of models for the algebraic $K$-theory of (unpointed) Waldhausen categories: via the (unpointed) $\oS$-construction and the 
$\infty$-categorical generalization of Thomason's construction/squares $K$-theory.

\begin{corollary}\label{unpointedS-vs-Sq}
For any unpointed Waldhausen category $(\C, \co(\C))$, we have 
\[\vert \oSun(\C, \co(\C))\vert \simeq \vert \Sq(\C, \co(\C), \C, \cocart)\vert.\]
\end{corollary}

In addition, we also obtain agreement with the usual (pointed) $\oS$-construction. 

\begin{proposition}
For any (pointed) Waldhausen category $(\C, \co(\C))$, the inclusion of the usual (pointed) $\oS$-construction is an equivalence after geometric realization:
\[ \vert \oS(\C, \co(\C))\vert  \xrightarrow\simeq  \vert \oSun(\C, \co(\C))\vert.\]   
\end{proposition}
\begin{proof}
By Corollary \ref{unpointedS-vs-Sq}, it suffices to show that the geometric realization of the pointed $\oS$-construction is canonically equivalent to the geometric realization of $\Sq(\C, \co(\C), \C, \cocart)$. This is shown in \cite{Wal} for ordinary Waldhausen categories, but the proof carries over to the $\infty$-categorical setting. Note that the space of $n$-simplices $\oS_n(\C, \co(\C))$ is equivalent to the space of $n$-strings of cofibrations starting at $*$. The inclusion of this space into the geometric realization of the vertical direction in fixed horizontal degree $n$:
$$\vert \Sq(\C, \co(\C), \C, \cocart)_{\bullet n}\vert$$
is an equivalence; the inverse functor is specified (on $0$-simplices) by the rule
\[ (c_0 \rightarrowtail \dots \rightarrowtail c_n) \mapsto (* \rightarrowtail c_1/c_0 \rightarrowtail \dots, \rightarrowtail c_n/c_0);\]
this defines a functor $\Sq(\C, \co(\C), \C, \cocart)_{\bullet n} \to \oS_n(\C, \co(\C))$ that extends to the geometric realization.
\end{proof}

\begin{remark}
After some evident modifications, the above results apply also to the case of Waldhausen categories with weak equivalences $(\C, \co(\C), w(\C))$ \cite{Bar4}, where $w(\C) \subset \C$ is a wide subcategory of weak equivalences (satisfying the glue axiom).
\end{remark}

\begin{remark}
Similar arguments apply to variations of the $\oS$-construction for other choices of `distinguished squares' instead of (co)cartesian ones. This is especially relevant for the case of the $K$-theory of higher homotopy categories where the distinguished squares are the respective higher weak pushouts (see \cite{Ra}). 
\end{remark}

\subsection{Comparison with the cobordism model for Waldhausen categories} We also discuss the connection with the ($\infty$-categorical extension of the) cobordism category model for $\K$-theory \cite{RS}. This may defined using the language of the present note 
as 
\[\Cob(\C, \co(\C)) := \vert \Span(\Sq(\C\op, \co(\C)\op, \C\op, \cocart\op)\vert\]
where the last condition means that we allow only cartesian squares in $\C\op$, i.e., cocartesian squares in $\C$. Thus, as a consequence of Theorem \ref{S-constr}, we obtain the following generalization of the main result of \cite{RS} (with a different proof). 

\begin{corollary}
For an unpointed Waldhausen $\infty$-category $(\C, \co(\C))$ we have 
\[\vert \Cob(\C, \co(\C)) \vert \simeq  \vert\oSun(\C, \co(\C))\vert.\]
\end{corollary}

\begin{proof}
 In view of the previous corollary and Theorem \ref{S-constr}, it suffices to prove 
 \[\vert \Sq(\C\op, \co(\C)\op, \C\op, \cocart\op) \vert \simeq \vert \Sq(\C, \co(\C), \C, \cocart)\vert.\]
Indeed the double Segal space on the left is just the opposite (in both directions) of the second, so their geometric realizations agree. 
\end{proof}

\subsection{Squares $K$-theory} In summary, using the equivalences for any $X \in \ssAn$ (Theorem \ref{S-constr})
$$\diag(X) \simeq \Stair(X) \simeq \Span(X\pop2),$$ 
combined with $\vert X \vert \simeq \vert X^{\pop2}\vert$, we conclude:

\begin{corollary}\label{main_cor}
For any $X \in \ssAn$, there are natural equivalences of spaces
\[\vert X \vert \simeq \vert \delta(X) \vert \simeq \vert \oS(X)\vert \simeq \vert\Q(X)\vert. \]
\end{corollary}

In view of the previous comparison results of models for algebraic $K$-theory, we may define $\K$-theory in the general setting of pointed double Segal spaces (or just bisimplicial spaces) as 
\[\K= \Omega \vert - \vert \colon \DSeg_* \to \An.\]
Indeed, this definition is a double $\infty$-categorical generalization of squares $K$-theory \cite{CKMZ}. Double $\infty$-categories seem to be a natural setting for $K$-theory because the abstract datum of a square enables us to formulate an inclusion-exclusion principle. Special instances of squares $K$-theory have recently been studied for ordinary (double) categories with weaker categorical properties than exact or Waldhausen categories or with different decomposition conditions, such as varieties (see e.g. \cite{Cam}) and manifolds \cite{HMMRS, MRS}.


\begin{thebibliography}{10}

\bibitem{Abe}
F.~Abell\'{a}n, \emph{Comparing lax functors of $(\infty, 2)$-categories}, arXiv (2023). \url{https://arxiv.org/abs/2311.12746} 

\bibitem{Ara}
K.~Arakawa, \emph{Classification Diagrams of Marked Simplicial Sets}, arXiv (2023). \url{https://arxiv.org/abs/2311.01101}

\bibitem{AC}
K.~Arakawa and B.~Cnossen, \emph{A short proof of the universality of the relative Rezk nerve}, Proc. Amer. Math. Soc. 154 (2026), no.~5, 1849--1853.

\bibitem{BDA} 
P.~Balmer and I.~Dell’Ambrogio, \emph{Mackey $2$-functors and Mackey $2$-motives}, EMS Monogr. Math., Z\"urich, European Mathematical Society (EMS), 2020.

\bibitem{Bar}
C.~Barwick, \emph{On the Q-construction for exact quasicategories}, arXiv (2013). \url{https://arxiv.org/abs/1301.4725}

\bibitem{Bar2}
C.~Barwick, \emph{On exact $\infty$-categories and the Theorem of the Heart}, Compos. Math. \textbf{151} (2015), no. 11, 2160--2186.

\bibitem{Bar4}
C.~Barwick, \emph{On the algebraic $K$-theory of higher categories}, J. Topol. \textbf{9} (2016), no.~1, 245--347.

\bibitem{Bar3} 
C.~Barwick, \emph{Spectral Mackey functors and equivariant algebraic $K$-theory. I.}, Adv. Math. \textbf{304} (2017), 646--727.

\bibitem{Cam}
J.~A.~Campbell, \emph{The $K$-theory spectrum of varieties}, Trans. Amer. Math. Soc. \textbf{371} (2019), no. 11, 7845--7884.

\bibitem{CKMZ}
J.~Campbell, J.~Kuijper, M.~Merling, and I.~Zakharevich, \emph{Algebraic $K$-theory for squares categories}, Ann. K-Th. \textbf{11} (2026), no. 1, 1--36.


\bibitem{CCL}
B.~Cnossen, T.~Lenz, and S.~Linskens, \emph{Parametrized higher semiadditivity and the universality of spans}, arXiv (2024). 
\url{https://arxiv.org/abs/2403.07676}

\bibitem{CCL2}
B.~Cnossen, T.~Lenz, and S.~Linskens, \emph{Universality of span $2$-categories and the construction of $6$-functor formalisms}, arXiv (2025). 
\url{https://arxiv.org/abs/2505.19192}


%\bibitem{FMa}
%W.~Fulton and R.~ MacPherson, \emph{Categorical framework for the study of singular spaces}, Mem. Amer. Math. Soc. \textbf{31} (1981), no. 243.

\bibitem{GR} 
 D.~Gaitsgory and N.~Rozenblyum, \emph{A study in derived algebraic geometry. Volume I: Correspondences and duality}, Math. Surv. Monogr., Vol. 221, Providence, RI, American Mathematical Society (AMS), 2017.

\bibitem{Harpaz}
Y.~Harpaz, \emph{Ambidexterity and the universality of finite spans}, Proc. of the London Math. Soc. \textbf{121} (2020), no. 5, 1121--1170.

\bibitem{Hau} 
R.~Haugseng, \emph{Iterated spans and classical topological field theories}, Math. Z. \textbf{289} (2018), no. 3, 1427--1488.

\bibitem{HHLN} 
R.~Haugseng, F.~Hebestreit, S.~Linskens, and J.~Nuiten, \emph{Two-variable fibrations, factorisation systems and $\infty$-categories of spans}, Forum Math. Sigma \textbf{11} (2023), no. e111, 1--70.


\bibitem{HS}
F.~Hebestreit and J.~Steinebrunner, \emph{A short proof that Rezk’s nerve is fully faithful}, Int. Math. Res. Not. IMRN \textbf{2025} (2025), no. 4, rnaf021.

\bibitem{HMMRS}
R.~S.~Hoekzema, M.~Merling, L.~Murray, C.~Rovi, and J.~Semikina, \emph{Cut and paste invariants of manifolds via algebraic $K$-theory}, Topology Appl. \textbf{316} (2022), no. 108105.

\bibitem{JT}
A.~Joyal and M.~Tierney, \emph{Quasi-categories vs. Segal spaces},  In: Categories in algebra, geometry and mathematical physics, Contemp. Math., vol. 431, pp. 277–326. Amer. Math. Soc., Providence, RI (2007).

\bibitem{Juran}
B.~Juran, \emph{On orthogonal factorization systems and double categories}, arXiv (2025). \url{https://arxiv.org/abs/2501.01363}. To appear in J. Pure Appl. Algebra.


\bibitem{LZ} Y.~Liu and W.~Zheng, \emph{Enhanced six operations and base change theorem for higher Artin stacks}, arXiv (2012). \url{https://arxiv.org/abs/1211.5948} (2024)

\bibitem{LR}
F.~Loubaton and J.~Ruit, \emph{On the squares functor and the Gaitsgory-Rozenblyum conjectures}, arXiv (2025). \url{https://arxiv.org/abs/2507.07807}

\bibitem{Mac} 
A.~W.~Macpherson, \emph{A bivariant Yoneda lemma and $(\infty, 2)$-categories of correspondences}, Algebr. Geom. Topol. \textbf{22} (2022), no. 6, 2689--2774.

\bibitem{Man} L.~Mann, \emph{A $p$-adic $6$-functor formalism in rigid-analytic geometry}, arXiv (2022). \url{https://arxiv.org/abs/2206.02022}

\bibitem{MG} 
A.~Mazel-Gee, \emph{The universality of the Rezk nerve}, Algebr. Geom. Topol \textbf{19} (2019), no. 7, 3217--3260.

\bibitem{MRS}
M.~Merling, G.~Raptis, and J.~Semikina, \emph{Parametrized scissors congruence $K$-theory of manifolds and cobordism categories}, arXiv (2025). \url{https://arxiv.org/abs/2504.01810}



\bibitem{Qui}
D.~Quillen, \emph{Higher algebraic K-theory. I.}, Algebraic K-theory, I: Higher K-theories (Proc. Conf., Battelle Memorial Inst., Seattle, Wash., 1972), pp. 85--147, Lecture Notes in Math. 341, Springer, Berlin 1973. 

\bibitem{Ra}
G.~Raptis, \emph{Higher homotopy categories, higher derivators, and {$K$}-theory}, Forum Math. Sigma \textbf{10} (2022), no. e54, 1--36.


\bibitem{RS}
G.~Raptis and W.~Steimle, \emph{A cobordism model for Waldhausen $K$-theory.} J. London Math. Soc. \textbf{99} (2019), no. 2, 516--534.


\bibitem{Ste} 
G.~Stefanich, \emph{Higher sheaf theory I: Correspondences}, arXiv (2020). arXiv:2011.03027(2020).

\bibitem{Wal}
F.~Waldhausen, \emph{Algebraic K-theory of spaces}, Algebraic and geometric topology (New Brunswick, N.J., 1983), 
pp. 318--419, Lecture Notes in Math. 1126, Springer, Berlin, 1985. 


\end{thebibliography}
\end{document}